\documentclass{article}

\usepackage[utf8]{inputenc}
\usepackage{amsmath,amssymb,amsthm, xcolor, enumitem, comment, todonotes}
\usepackage{url, hyperref}
\usepackage[all]{xy}
\usepackage[english]{babel}

\newtheorem{definition}{Definition}[section]

\newtheorem{question}[definition]{Question}

\newtheorem{remark}[definition]{Remark}

\newtheorem{theorem}{Theorem}[section]
\newtheorem{claim}[definition]{Claim}

\newtheorem{proposition}[definition]{Proposition}
\newtheorem{lemma}[definition]{Lemma}
\newtheorem{corollary}[definition]{Corollary}

\newtheorem{Lemma}[definition]{Lemma}

\newcommand{\po}{\mathbb{P}}
\newcommand{\qo}{\mathbb{Q}}
\newcommand{\la}{\langle}
\newcommand{\ra}{\rangle}

\newcommand{\name}{\dot}

\newcommand{\uhr}{\upharpoonright}

\newcommand{\vc}{\mathcal{V}}
\newcommand{\wc}{\mathcal{W}}

\DeclareMathOperator{\cf}{cof}

\begin{document}

\title{A note on the Ketonen order and Lipschitz reducibility between ultrafilters}
\author{Eyal Kaplan\\UC Berkeley\\Berkeley, CA}

\date{\today}
\maketitle

\begin{abstract}
In his study of the Ultrapower Axiom (UA), Goldberg revealed a connection between UA and the determinacy of certain games that witness Lipschitz reducibility between ultrafilters. In particular, he analyzed the relationship between the Ketonen and Lipschitz orders—two natural extensions of the Mitchell order from normal measures to arbitrary $\sigma$-complete ultrafilters—and proved that the Lipschitz order extends the Ketonen order. He further observed that under UA the two orders coincide. Goldberg asked if it's consistent that the orders differ from each other. We show that the answer is positive. In fact, even the Weak Ultrapower Axiom does not imply that the Ketonen and Lipschitz orders coincide.
\end{abstract}

\section{Introduction}

The study and analysis of the various connections between $\sigma$-complete ultrafilters plays a central role in the theory of large cardinals. Building on this perspective, Goldberg introduced and studied the Ultrapower Axiom (UA), which asserts
that for every pair of $\sigma$-complete ultrafilters, $U,W$, there are $\sigma$-complete ultrafilters $W^*\in M_U$ and $U^*\in M_W$ such that $M^{M_U}_{W^*} =M^{M_W}_{U^*}$ and $j^{M_U}_{W^*}\circ j_U = j^{M_{W}}_{U^*}\circ j_{W}$.\footnote{For a $\sigma$-complete ultrafilter $U$, $M_U$ denotes its ultrapower model, and $j_U\colon V\to M_U$ denotes the corresponding elementary embedding. If $U\in N$ for some inner model $N$ of $V$, we denote by $(M_U)^N$ and $j^{N}_U \colon N\to (M_U)^N$ the corresponding ultrapower and elementary embedding over $N$.} 
Goldberg’s analysis of UA has led to a series of deep structural results about the set-theoretic universe. One key discovery is that, under UA, the class of $\sigma$-complete ultrafilters is well-ordered by a natural ordering called the Ketonen order (see Definition~\ref{Def: Ketonen order} below). Goldberg further showed that UA is equivalent to the linearity of the Ketonen order.

In his work, Goldberg identified that Ketonen comparability between ultrafilters implies determinacy of certain infinite games, reminiscent of the Lipschitz order used in descriptive set theory to compare subsets of the Cantor space (see Definition~\ref{Def: Lipschitz order on ultrafilters} below). He proved that, under UA, the Ketonen order and the Lipschitz order coincide when restricted to $\sigma$-complete ultrafilters, and raised the question whether it is consistent for the two orders to disagree (see \cite[Question 9.2.10]{GoldbergUABook}). Our main result gives a positive answer to this question:

\begin{theorem}\label{theorem: intro}
	Consistently from a measurable cardinal, there are two $\sigma$-complete ultrafilters $\mathcal{V}, \mathcal{W}$ such that $\mathcal{V}$ is Lipschitz below $\mathcal{W}$, but $\mathcal{V}$ and $\mathcal{W}$ are Ketonen-incomparable. Furthermore, the same is also consistent with the assumption that the Weak Ultrapower Axiom holds.
\end{theorem}

The \textit{Weak Ultrapower Axiom} is obtained from UA by removing the requirement that $j^{M_U}_{W^*}\circ j_U = j^{M_{W}}_{U^*}\circ j_{W}$. Models of the Weak UA+$\neg$UA were constructed in \cite{benhamouGoldberg2024applications} and more recently in \cite{kaplan2025number}. We will rely on the construction from \cite{kaplan2025number} in the proof of Theorem \ref{theorem: intro}.

The structure of this paper is as follows: 
\begin{itemize}
	\item In Section 2 we define the Ketonen and Lipschitz orders and outline their basic properties; all the results in this section are due to Goldberg. 
	\item In Section 3 we show that, consistently, the Ketonen and Lipschitz orders do not coincide (however, this is not done in a model of the Weak UA). See Theorem \ref{Theorem: Separating Lipschitz from Ketonen the Easy way}. 
	\item In Section 4 we sketch the basic properties of forcing with  nonstationary support products.
	\item In Section 5, we separate the Ketonen and Lipschitz orders in a model of the Weak UA, thereby completing the proof of Theorem \ref{theorem: intro}.
\end{itemize}

\textbf{Acknowledgments:} The author deeply thanks Gabe Goldberg and Eilon Bilinsky for many conversations on the subject of this paper.

\section{Preliminaries - the Ketonen and Lipschitz orders}

In this section, we present the Ketonen and Lipschitz orders and draw the connections between them. All of the results in this section are due to Goldberg, and can be found in detail in \cite{GoldbergUABook}.

\begin{definition}[The Ketonen order\footnote{The modern formulation of the Ketonen order is due to Goldberg, building on Ketonen’s earlier work, which treated only weakly normal ultrafilters \cite{Ketonen1972strongcompactnesscardinalsins}.}]\label{Def: Ketonen order}
	Let $U,W$ be $\sigma$-complete ultrafilters. We say that $U$ is Ketonen below $W$, and denote $U<_k W$, if and only if one of the following equivalent conditions hold:
	\begin{enumerate}
		\item There exists $I\in W$ and a sequence $\la U_\xi \colon \xi\in I \ra$ of $\sigma$-complete ultrafilters, such that each $U_\xi$ concentrates on $\xi$, and, for every $X\subseteq \kappa$ (where $\kappa$ is the underlying ordinal of $U$), 
		$$X\in U \iff \{ \xi\in I \colon X\cap \xi \in U_\xi \}\in W.$$
		\item There exists a $\sigma$-complete ultrafilter $U^*\in M_W$ and an elementary embedding $k\colon M_U\to M^{M_W}_{U^*}$ such that $k\circ j_U = j^{M_W}_{U^*}\circ j_W$, and $k\left( [Id]_U \right)< j_{U^*}\left( [Id]_W \right)$.
	\end{enumerate}

\end{definition}

For the equivalence between the definitions, see \cite[Lemma 3.3.4]{GoldbergUABook}. The Ketonen order is a strict well-founded order on the class of all $\sigma$-complete ultrafilters (see \cite[Subsection 3.3.2]{GoldbergUABook}). When restricted to normal measures, the Ketonen order coincides with the Mitchell order (see \cite[Theorem 3.4.1]{GoldbergUABook}, or, alternatively, Corollary \ref{Corollary: Ketonen and Lipschitz are Mitchell on normal measures} below). 

The study of the Ketonen order is especially interesting in the context of Goldberg's Ultrapower Axiom. Recall that the \textit{Ultrapower Axiom} (UA) is the assertion that for every pair of $\sigma$-complete ultrafilters, $U,W$, there are $\sigma$-complete ultrafilters $U^*\in M_W$ and $W^*\in M_U$ such that $M^{M_U}_{W^*} =M^{M_W}_{U^*}$, and $j^{M_U}_{W^*}\circ j_U = j^{M_{W}}_{U^*}\circ j_{W}$. The \textit{Weak Ultrapower Axiom} is obtained from UA by removing the requirement that $j^{M_U}_{W^*}\circ j_U = j^{M_{W}}_{U^*}\circ j_{W}$.

\begin{theorem}[Goldberg, {\cite[Theorem 3.5.1]{GoldbergUABook}}]\label{Theorem: Goldberg, UA is equivalent to the linearity of Ketonen}
	$\mbox{UA}$ is equivalent to the linearity of the Ketonen order.
\end{theorem}

Thus, under the UA, the class of all $\sigma$-complete ultrafilters is well ordered by the Ketonen order. Furthermore, every such ultrafilter is ordinal definable via its rank with respect to the Ketonen order.

We proceed and define the Lipschitz order between $\sigma$-complete ultrafilters. For that, we describe the Lipschitz game of length $\kappa$, $G_{\kappa}(W,U)$, where $W,U$ are subsets of $\mathcal{P}(\kappa)$ for an ordinal $\kappa$. The game is being held between two players, I and II, and consists of $\kappa$ stages. On the $i$-th stage ($i<\kappa$), Player I chooses a bit $a(i)\in \{0,1\}$, and Player II responds with a bit $b(i)\in \{0,1\}$. Since Player I moves first at any stage, they are aware of the sequences $\la a(j) \colon j<i \ra, \la b(j) \colon j<i \ra$ constructed so far; Player II moves second, being aware of the values of $\la a(j) \colon j\leq i  \ra, \la b(j) \colon j<i \ra$. After $\kappa$ rounds, the players have constructed a pair of subsets of $\kappa$,
$$ a = \{ i<\kappa \colon a(i) = 1 \} $$
$$ b = \{ i<\kappa \colon b(i) = 1 \}. $$
Player II wins if $a\in W\leftrightarrow b\in U$. Otherwise, Player I wins. 

The (strict) Lipschitz order is defined by setting $U<_L W$ if and only if Player I has a winning strategy in both games $G_{\kappa}(W,U)$ and $G_{\kappa}(\mathcal{P}(\kappa)\setminus W, U)$.

An equivalent definition of the Lipschitz order may be given in terms of Lipschitz functions. For an ordinal $\kappa$, we say that a function $f\colon \mathcal{P}(\kappa)\to \mathcal{P}(\kappa)$ is a super-Lipschitz function if for every $\xi<\kappa$ and $A\subseteq \kappa$, the value of $f(A)\cap (\xi+1)$ depends only on $A\cap \xi$, in the sense that--
$$ f(A)\cap (\xi+1) = f(A\cap \xi)\cap (\xi+1).$$
We say that $U$ is super-Lipschitz reducible to $W$ if there exists a super-Lipschitz function $f\colon \mathcal{P}(\kappa)\to \mathcal{P}(\kappa)$ such that $f^{-1}[W] = U$. We call such $f$ a super-Lipschitz reduction of $U$ to $W$.

\begin{Lemma}
	Player I has a winning strategy in $G_{\kappa}(W,U)$ if and only if $U$ is super-Lipschitz reducible to $\mathcal{P}(\kappa)\setminus W$. In the case where $W$ is an ultrafilter, this is equivalent to $U$ being super-Lipschitz reducible to $W$.
\end{Lemma}

\begin{proof}
	It's not hard to check that whenever $W$ is an ultrafilter, a function $f$ is a super-Lipschitz reduction of $U$ to $W$ if and only if the function $f^*(A) =\kappa\setminus f( A)$ is a super-Lipschitz reduction of $U$ to $\mathcal{P}(\kappa)\setminus W$. Thus, we only need to prove the first part of the lemma.
	
	Assume that Player I has a winning strategy in $G_{\kappa}(W,U)$. Given $A\subseteq \kappa$, let $\la b^A(i)  \colon i<\kappa\ra$ be the characteristic function of $A$, namely, for every $i<\kappa$, $b^A(i) = 1 \leftrightarrow i\in A $. Let $\la a^A(i) \colon i<\kappa \ra$ be the sequence of moves of Player I, when played according to their winning strategy, where Player II plays $b^A(i)$ on their $i$-th move (for clarity, we stress that $a^A(i)$ is the move on the $i$-th round, and we assume that $a^A(0) = 1$). For every $A\subseteq \kappa$, let $f(A) = \{ i<\kappa \colon a^A(i) = 0 \}$. Note that $f$ is a super-Lipschitz function. Let us argue that $U = f^{-1}[\mathcal{P}(\kappa)\setminus W]$. Indeed, since the strategy is winning for Player I, we have for every $A\subseteq \kappa$, $$A\in U \leftrightarrow \{ i<\kappa \colon b^A(i) = 1 \}\in U \leftrightarrow \{ i<\kappa \colon a^{A}(i) = 1 \}\notin W\leftrightarrow f(A)\in \mathcal{P}(\kappa)\setminus W.$$
	For the other direction, assume that $f\colon \mathcal{P}(\kappa)\to \mathcal{P}(\kappa)$ is super-Lipschitz and $U = f^{-1}[\mathcal{P}(\kappa)\setminus W]$. Let us define a winning strategy for Player I in $G_{\kappa}(W,U)$. Assume that $i<\kappa$ and the players have  completed $i$ rounds, constructing sequences $\la a(j) \colon j<i \ra, \la b(j) \colon j<i \ra$. In the $i$-th round, Player I checks whether $i\in f(\{ j<i \colon b(j) = 1 \})$. Player I chooses $a(i) = 1$ if and only if the answer is positive. This is a winning strategy for Player I. Indeed, assume that $\la a(i) \colon i<\kappa \ra, \la b(i) \colon i<\kappa \ra$ are the full sequences of moves of Players I, II, respectively. Then:
	\begin{align*}
		\{ i<\kappa \colon a(i) = 1 \}\notin W \leftrightarrow& \  \{i<\kappa \colon i\in f\left(  \{ j<\kappa \colon b(j) = 1  \}\cap i \right) \}\notin W \\  \leftrightarrow & \  f\left( \{j<\kappa \colon b(j) = 1\} \right) \notin W
		\\ \leftrightarrow &\ \{ j<i \colon b(j) = 1 \}\in U.
	\end{align*}
\end{proof}

In the light of the previous lemma, we define:

\begin{definition}[The Lipschitz order on ultrafilters]\label{Def: Lipschitz order on ultrafilters}
Let $U,W\subseteq \mathcal{P}(\kappa)$ be ultrafilters on $\kappa$. We say that $U$ is Lipschitz below $W$, and denote $U<_L W$, if one of the following equivalent conditions hold:
\begin{itemize}
	\item player I has a winning strategy in $G_{\kappa}(W,U)$.
	\item there exists a super-Lipschitz function $f\colon \mathcal{P}(\kappa)\to \mathcal{P}(\kappa)$ that reduces $U$ to $W$, in the sense that $f^{-1}[W] = U$.
\end{itemize}

\end{definition}

The Lipschitz order is a strict partial order on $\sigma$-complete ultrafilters (see \cite[Corollary 3.4.29]{GoldbergUABook}). To the best of our knowledge, it is still not known whether the Lipschitz order on $\sigma$-complete ultrafilters must be well-founded (but it is well-founded under UA, see Corollary \ref{Corollary: UA implies that Ketonen is Lipschitz} below). We remark that the Lipschitz order on arbitrary subsets of $\mathcal{P}(\kappa)$ is ill-founded in models of AC (see \cite[Theorem 3.4.30]{GoldbergUABook}); however, in this paper, we focus only on the Lipschitz order restricted to $\sigma$-complete ultrafilters. We freely refer to it as the “Lipschitz order,” without mentioning the restriction to pairs of ultrafilters.
We proceed and prove that the Lipschitz order extends the Ketonen order. We provide a proof directly from the definitions given above (for an alternative justification, see Proposition \ref{Proposition: equivalent characterization of Lipschitz and Ketonen}). 

\begin{proposition}[Goldberg\footnote{A similar result was independently observed by Eilon Bilinsky and the author, based on an equivalent definition of the Lipschitz order; the equivalence between the definitions was observed by Goldberg.}]\label{Lemma: Ketonen implies Lipschitz}
	Assume that $U,W$ are $\sigma$-complete ultrafilters on $\kappa$. Then $U<_k W$ implies $U<_L W$.
\end{proposition}

\begin{proof}
	Let $I\in W$ and $\la U_i \colon i\in I \ra$ be a sequence of $\sigma$-complete ultrafilters, each $U_i$ concentrates on $i$, such that $ U = \{  X\subseteq \kappa \colon \{i\in I \colon X\cap i \in U_i \}\in W \}$. 
	
	Let us define a strategy for Player I in $G_\kappa(U,W)$ (alternatively, one can check that the function $f\colon \mathcal{P}(\kappa)\to \mathcal{P}(\kappa)$ defined by $f(X) = \{ i\in I \colon X\cap i\in U_i \}$ is super-Lipschitz and $f^{-1}[W] = U$). Assume that $i<\kappa$, and the players constructed the sequences $\la a(j) \colon j<i \ra$ and $\la b(j) \colon j<i \ra$. Player I then checks if $i\in I$; if not, they play as they wish. If $i\in I$, Player I checks if $\{ j<i \colon b\left( j \right) = 1 \}\in U_{i}$, and plays $a(i)=0$ if the answer is positive, and $a(i)=1$ otherwise.
	
	We argue that the above describes a winning strategy for Player I. Let $\vec{a} = \la a(i) \colon i<\kappa \ra, \vec{b} = \la b(i) \colon i<\kappa \ra$ be the resulting sequence of moves. Then, indeed, 
	\begin{align*}
		 \{ i<\kappa \colon a(i)=1 \}\in W 	\iff & \{  i\in I \colon a(i)=1  \}\in W\\
		\iff &\{  i\in I \colon \{ j<i \colon b(j)=0 \}\in U_i  \}\in W \\
		\iff & \{ j<\kappa \colon b(j)=0 \}\in U. 
	\end{align*}	
\end{proof}

Since the UA is equivalent to the linearity of the Ketonen order, we can immediately deduce the following:

\begin{corollary}[Goldberg]\label{Corollary: UA implies that Ketonen is Lipschitz}
	Assume UA. Then the Ketonen order and the Lipschitz order coincide, and both are linear.
\end{corollary}

We conclude this section by mentioning alternative characterizations of the Ketonen and Lipschitz orders, due to Goldberg. This characterization yields a simpler proof of Proposition \ref{Lemma: Ketonen implies Lipschitz}.

\begin{definition}
	Let $\kappa$ be an ordinal. A set $Z\subseteq \mathcal{P}(\kappa)$ concentrates on a set $S\subseteq \kappa$ if for every $X,Y\subseteq \kappa$, if $X\cap S = Y\cap S$ then $X\in Z\leftrightarrow Y\in Z$.
\end{definition}

\begin{proposition}[Goldberg] \label{Proposition: equivalent characterization of Lipschitz and Ketonen}
	Let $U,W$ be $\sigma$-complete ultrafilters on some ordinal $\kappa$. 
	\begin{enumerate}
		\item $U<_L W$ if and only if there exists $Z\in M_W$ such that $Z$ concentrates on $[Id]_W$ and for every $A\subseteq \kappa$, $A\in U \leftrightarrow j_W(A)\cap \delta\in Z$.
		\item $U<_k W$ if and only if there exists $Z\in M_W$ such that $Z$ is a $\sigma$-complete ultrafilter concentrating on $[Id]_W$, and for every $A\subseteq \kappa$, $A\in U \leftrightarrow j_W(A)\cap \delta\in Z$.
	\end{enumerate}
\end{proposition}

\begin{proof}
	$\ $
	\begin{enumerate}
		\item Assume $U<_L W$. Let $f\colon \mathcal{P}(\kappa)\to \mathcal{P}(\kappa)$ be a super-Lipschitz function such that $f^{-1}[U] = W$.  In $M_W$, define $Z = \{  X\subseteq \left[Id\right]_W \colon  [Id]_W \in j_W(f)(X) \}$. The fact that $j_W(f)$ is super-Lipschitz implies that $Z$ concentrates on $[Id]_W$. Note that for every $A\subseteq \kappa$,
		\begin{align*}
			A\in U \leftrightarrow & f(A)\in W \\
			\leftrightarrow & [Id]_W\in j_W(f)(j_W(A))\\ \leftrightarrow & \left[Id\right]_W\in j_W(f)(j_W(A)\cap [Id]_W) \\
			\leftrightarrow & j_W(A)\cap[Id]_W\in Z.
		\end{align*}
		For the other direction, assume that $Z\in M_W$ concentrates on $[Id]_W$, and for every $A\subseteq \kappa$, $A\in U \leftrightarrow j_W(A)\cap [Id]_W\in Z$. Let $g\colon \kappa\to \mathcal{\kappa}$ be a function such that $[g]_W = Z$. Since $Z$ concentrates on $[Id]_W$, we can assume that for every $\alpha<\kappa$, $g(\alpha)$ concentrates on  $\alpha$. Define $f\colon \mathcal{P}(\kappa)\to \mathcal{P}(\kappa)$ by setting, for each $A\subseteq \kappa$, $f(A) = \{ \alpha<\kappa \colon A\cap \alpha\in g(\alpha)  \}$. Then $f$ is super-Lipschitz, and for every $A\subseteq \kappa$, 
		\begin{align*}
			A\in U \leftrightarrow j_W(A)\cap [Id]_W\in Z \leftrightarrow \{ \alpha<\kappa \colon A\cap \alpha\in g(\alpha) \} \in W \leftrightarrow f(A)\in W.
		\end{align*}
		\item Assume that $U<_k W$, and let $I\in W$ and $\la U_i \colon i\in I \ra$ witness this, each $U_i$ concentrates on $i$. Let $Z = [i\mapsto U_i]_W$. Then $Z\in M_W$ is a $\sigma$-complete ultrafilter that concentrates on $[Id]_W$, and, given $A\subseteq \kappa$,
		
		$$ A\in U \leftrightarrow \{ i\in I \colon A\cap i \in U_i  \}\in W \leftrightarrow j_W(A)\cap \delta\in Z. $$
		
		For the other direction, just take $\la U_i \colon i<\kappa \ra$ to be a sequence of $\sigma$-complete ultrafilters representing $Z$, in the sense that $[i\mapsto U_i]_W = Z$, and take $I\in W$ to be a set such that for each $i\in I$, $U_i$ concentrates on $i$.

	\end{enumerate}
\end{proof}

A simple corollary of Proposition \ref{Proposition: equivalent characterization of Lipschitz and Ketonen} is that, when restricted to normal measures, the Lipschitz and Ketonen orders coincide with the Mitchell order\footnote{Recall that for normal measures $U,W$, we say that $U$ is Mitchell below $W$ and denote $U\vartriangleleft W$, if $U\in M_W$}:

\begin{corollary}\label{Corollary: Ketonen and Lipschitz are Mitchell on normal measures}
	Assume that $U,W$ are normal measures on $\kappa$. Then $U\vartriangleleft W \leftrightarrow U<_k W \leftrightarrow U<_L W$.
\end{corollary}

\begin{proof}
	It suffices to prove that $U\vartriangleleft W \to U<_k W$ and $U<_L W \to U\vartriangleleft W$. For the former, note that $U\in M_W$ is a $\sigma$-complete ultrafilter that concentrates on $\kappa = [Id]_W$, and for every $A\subseteq \kappa$, $j_W(A)\cap \kappa = A$, so $U<_k W$. Thus, let us concentrate on the proof that $U <_L W$ implies $U\vartriangleleft W$. By Proposition \ref{Proposition: equivalent characterization of Lipschitz and Ketonen}, we can assume that for some $Z\in M_W$ that concentrates $\kappa = [Id]_W$, and for every $A\subseteq \kappa$, $A\in U \leftrightarrow j_W(A)\cap \kappa\in Z$. In particular, $U = Z\cap \mathcal{P}(\kappa)$. It follows that $U\in M_W$. 
\end{proof}

\section{Separating the Ketonen and Lipschitz orders}

Assume GCH. Let $\kappa$ be a measurable cardinal, and denote 
$$I = \{ \alpha<\kappa \colon \alpha \text{ is inaccessible} \}.$$
Let $\la \po_\alpha, \dot{\qo}_\alpha \colon \alpha<\kappa \ra$ be an Easton support product in which, for each $\alpha\in I$, $\qo_\alpha = \{  0_{\qo_\alpha}, 0,1 \}$, where $0, 1$ are incompatible elements (we simply denote each forcing $\qo_\alpha$ by $\qo$). For every $\alpha\in \kappa\setminus I$, $\qo_\alpha$ is the trivial forcing. In other words,
$$ \po = \{ f\colon X\to 2 \colon X\subseteq I \text{ is an Easton set} \}$$
ordered by inclusion. Here, by an “Easton set”, we mean a set $X\subseteq \kappa$ such that, for every $\lambda\in I\cup \{\kappa\}$, $X\cap \lambda$ is bounded in $\lambda$.

Let $G\subseteq \po$ be generic over $V$. The following is standard:

\begin{claim}
	$\kappa$ remains measurable in $V[G]$. Furthermore, whenever $U\in V$ is a normal measure on $\kappa$,
	\begin{enumerate}
		\item $j_U(\po)$ factors to the form $\po\times \qo \times j_U(\po)\setminus (\kappa+1)$, where $j_U(\po)\setminus (\kappa+1) = \{  f\colon X\to 2 \colon X\subseteq j_U(I)\setminus (\kappa+1) \text{ is an Easton set} \}$.
		\item There exists a set $G^*\in V[G]$ such that, for every $i\in \{0,1\}$, $H_i =  G\times \{i\}\times G^* $ is $j_U(\po)$-generic set over $M_U$.
		\item There are normal measures $U_0, U_1\in V[G]$ on $\kappa$ with corresponding ultrapower embedding $j^{V[G]}_{U_i}\colon V[G]\to M_U[H_i]$ mapping $G$ to $H_i$.
	\end{enumerate}
\end{claim}

$G^*$ as above is easily constructed in $V[G]$ by enumerating all the antichains of $j_U(\po)\setminus (\kappa+1)$ (belonging to $M_U[G]$, or, equivalently, $V[G]$) and constructing an ascending $\kappa^+$-sequence of conditions that meets them all. Since $G^*$ is generic over $M_U[G]$, it's also generic over $M_U[G\times \{i \}]$ for every $i\in \{0,1\}$. Thus each $H_i$ is generic over $M_U$ for $j_U(\po)$.

We are now ready to present one of the main results of this paper, which is a separation between the Ketonen order and the Lipschitz order in $V[G]$. We will make use of the following standard claim:

\begin{claim}\label{Claim: HayutPovedaIntersectionOfImagesOfClubs}
	Let $m\geq 1$ and $U$ a normal measure on a measurable cardinal $\kappa$. Then:
	$$\bigcap\{ j_{U^m}(C) \colon C\in V \mbox{ is a club in }\kappa \}= \{ j_{U^n}(\kappa) \colon n< m \}.$$
\end{claim}
The proof can be found, for instance, in \cite[Lemma 5.3]{HayutPoveda2022gluing}. For the sake of completeness, we will present the proof of Claim \ref{Claim: HayutPovedaIntersectionOfImagesOfClubs} below, but we defer it until after the proof of Theorem \ref{Theorem: Separating Lipschitz from Ketonen the Easy way}.

\begin{theorem}\label{Theorem: Separating Lipschitz from Ketonen the Easy way}
	Assume $V= L[U]$, $\kappa$ is the unique measurable cardinal, $\po$ is the forcing notion described above and $G,G^*, U_0,U_1$ are as above. Denote $\vc = U_0$ and $\wc$ the measure derived from $j^{V[G]}_{(U_1)^2}$ using the ordinal $j_U(\kappa)+\kappa$ as a seed. Then $\vc<_{L} \wc$ but $\vc, \wc$ are Ketonen incomparable.
\end{theorem}

\begin{proof}
	We first argue that $\vc<_L \wc$, by using the equivalent characterization of the Lipschitz order given in Proposition \ref{Proposition: equivalent characterization of Lipschitz and Ketonen}. For an intuitive explanation of this argument, see the first few paragraphs following the proof of Theorem \ref{Theorem: Separating Lipschitz from Ketonen the Easy way}. 
	
	Fix in advance a well order $\vartriangleleft$ on $H_{\kappa^+}$. Denote by $\mathcal{A}$ the set of antichains in $\po$. Note that $\po$ is $\kappa-c.c.$, so $\mathcal{A}\subseteq V_\kappa$.
	
	We would like to define a set $Z\in \text{Ult}(V[G], \wc)$ concentrating on $j_U(\kappa) \leq [Id]_\wc$, such that, for every $X\subseteq \kappa$, $X\in \vc$ if and only if $j_{\wc}(X)\cap j_U(\kappa)\in Z$. For that, we need to establish some notations. 
	
	 Given a $\po$-name $\dot{X}$ for a subset of $\kappa$, recall that it has an associated “nice name” of the form $\bigcup_{i<\kappa} \left( \{ \check{i} \}\times A_{i} \right)$; this means that, for each $i<\kappa$, $A_i\in \mathcal{A}$ is a maximal antichain in $\{ q\in \po \colon q\Vdash \check{i}\in \dot{X} \}$ (we allow $A_i = \emptyset$). In particular, $(\dot{X})_G$ can be retrieved from $\vec{A} = \la A_i \colon i<\kappa \ra$ and $G$, since $(\dot{X})_G = \text{val}\left( \vec{A}, G\right)$, where
	$$ \text{val}\left( \vec{A}, G \right) = \{  i< \text{lh}(\vec{A}) = \kappa \colon G\cap A_{i}\neq \emptyset \}. $$

	The above shows that for every $X\subseteq \kappa$ in $V[G]$ there is $\vec{A}\in {\mathcal{A}}^\kappa$ (in $V$) such that $X = \text{val}(\vec{A}, G)$.

	Note that $\text{Ult}(V[G], \wc) = \text{Ult}(V[G], (U_1)^2)$ has the form $M_{U^2}[G\times \{1\}\times G^*\times \{1\}\times G^{**}]$ where $G^{**}\in M_U[G\times \{1\}\times G^*]$ is $j_{U^2}(\po)\setminus (j_U(\kappa)+1)$-generic over $M_{U^2}[G\times \{1\}\times G^*]$. 
	
	Denote $\kappa_1 = j_U(\kappa)$. Fix $Y\subseteq \kappa_1$ in $M_{U^2}[G\times \{1\}\times G^*\times \{1\}\times G^{**}]$. Since $j_{U^2}(\po)\setminus \kappa_1$ is sufficiently closed,  $Y\in M_{U^2}[G\times \{1\}\times G^*]$. Thus, there exists $\vec{A}\in \left( j_{U^2}(\mathcal{A})\right)^{\kappa_1}$ such that 
	$$M_{U^2}[G\times \{1\}\times G^*]\vDash Y = \text{val}(\vec{A}, G\times \{1\}\times G^*).$$
	Denote the $j_{U^2}(\vartriangleleft)$-least such $\vec{A}\in M_{U^2}$ by $\vec{A}_Y$.
	
	In $\text{Ult}(V[G], \wc) = M_{U^2}[G\times \{1\}\times G^*\times \{1\}\times G^{**}]$, define:
	\begin{align*}
		Z &= \{  Y\subseteq \kappa_1 \colon \kappa\in \text{val}(\vec{A}_Y, G\times \{ 0 \} \times G^* ) \}\\
		& = \{  Y\subseteq \kappa_1 \colon (G\times \{0\}\times G^*) \cap \vec{A}_Y(\kappa) \neq \emptyset \}.
	\end{align*}
	%
	Clearly $Z$ concentrates on $\kappa_1$. Thus, in order to establish that $\vc<_L \wc$, it suffices to prove the following:
	\begin{claim}
		For every $X\subseteq \kappa$ in $V[G]$, $ X\in \vc \iff  j_\wc(X)\cap \kappa_1\in Z .$
		
	\end{claim}

	\begin{proof}
		Fix such $X\in V[G]$. In $V[G]$, let $\vec{A} = \la A_i \colon i<\kappa \ra\in V$ be the $\vartriangleleft$-least sequence in $(\mathcal{A})^\kappa$ such that $X = \text{val}(\vec{A}, G)$.
		
		Denote $Y = j_{\wc}(X)\cap \kappa_1$. It suffices to prove that, in the above notations, $\vec{A}_Y = j_U(\vec{A})$. Indeed, this implies:
		\begin{align*}
			X\in \vc & \iff  \kappa\in j_\vc(X) \\
			& \iff \kappa\in \text{val}\left( j_{U}(\vec{A}), G\times \{ 0 \}\times G^* \right) \\
			& \iff \kappa\in \text{val}\left( \vec{A}_Y, G\times \{ 0 \}\times G^* \right) \\
			& \iff Y\in Z.
		\end{align*}
		As desired.
		
		Thus, we proceed to prove that $\vec{A}_Y = j_U(\vec{A})$. First, recall that $ Y = j_{\wc}(X)\cap \kappa_1$, and, since $\wc \equiv_{RK} (U_1)^2$, $Y = j_{U_1}(X)$. Therefore, by elementarity,
		$$M_U[G\times\{1\}\times G^*] \vDash Y = \text{val}\left( j_U(\vec{A}), G\times\{1\}\times G^* \right). $$
		Since $j_U(\vec{A}) = j_{U^2}(\vec{A})\uhr \kappa_1\in M_{U^2}$ and the $"\text{val}"$ computation is absolute, 
		$$ M_{U^2}[G\times \{1\}\times G^*]\vDash Y = \text{val}\left( j_U(\vec{A}), G\times\{1\}\times G^* \right). $$
		In order to deduce that $\vec{A}_Y = j_U(\vec{A})$ we need to prove $j_{U^2}(\vartriangleleft)$-minimality of $j_U(\vec{A})$. Assume that some $\vec{B}\in \left( j_{U^2}(\mathcal{A}) \right)^{\kappa_1}$ is  $j_{U^2}(\vartriangleleft)$-below $j_U(\vec{A})$, and
		$$ M_{U^2}[G\times \{1\}\times G^*]\vDash Y = \text{val}(\vec{B}, G\times \{1\}\times G^*). $$
		Since $j_{U^2}(\vartriangleleft) \uhr \left( H_{(\kappa_1)^+} \right)^{M_U} = j_U(\vartriangleleft)$ and $j_U(\vec{A})\in \left( H_{(\kappa_1)^+} \right)^{M_U}$, we deduce that $\vec{B}\in M_U$, $\vec{B}$ is $j_U(\vartriangleleft)$-below $j_U(\vec{A})$, and 
		$$ M_U[G\times \{1\}\times G^*] \vDash j_{U_1}(X) = \text{val}( \vec{B}, G\times \{1\}\times G^* ). $$
		However, this contradicts elementarity and the fact that, in $V[G]$, $\vec{A}$ is $\vartriangleleft$-least such that $X = \text{val}(\vec{A}, G)$.
	\end{proof}

	Next, we argue that $\vc, \wc$ are Ketonen incomparable. Since we already established that $\vc <_{L} \wc$, we only need to rule out the possibility that $\vc {<_k} \wc$. Assume otherwise. Let $U^*\in \text{Ult}(V[G], \wc)$ and $k\colon M_{\vc}\to M^{M_{\wc}}_{U^*}$ be an elementary embedding such that $k\circ j_{\vc} = j^{M_{\wc}}_{U^*}\circ j_{\wc}$ and $k([Id]_{\vc})< j_{U^*}\left(  [Id]_{\wc} \right)$, namely $k(\kappa) < j_{U^*}\left( \kappa_1+\kappa \right) $. We will derive a contradiction below by showing that 	$$\left(k\circ j_{\vc}\right)(G )\neq \left(j^{M_\wc}_{U^*}\circ j_{\wc}\right)( G ).$$ 
	
	Since we forced over $V = L[U]$, we can assume that the ultrapower embedding $(k\circ j_{\vc} )\uhr L[U] = (j^{M_{\wc}}_{U^*}\circ j_{\wc} )\uhr L[U]$ is a finite iterated ultrapower $j_{U^m}$ of $L[U]$ (the iteration is finite since $\po$ is $\sigma$-closed). Note that for every club $C\subseteq \kappa$ in $L[U]$, $ \kappa\in j_\vc(C) $, and so $ k(\kappa)\in k( j_\vc(C) ) = j_{U^m}(C) $. By Claim \ref{Claim: HayutPovedaIntersectionOfImagesOfClubs}, it follows that $k(\kappa) = j_{U^i}(\kappa)$ for some $i< m$. Since $k(\kappa)< j_{U^*} (\kappa_1+\kappa) = \kappa_1+\kappa$, we deduce $k(\kappa) \in \{ \kappa, \kappa_1 \}$. Note that we used here the fact that the only measurable cardinal in $\text{Ult}(V[G], \wc)$ is $j_{U^2}(\kappa)$, so $\text{crit}(j^{M_\wc}_{U^*})> \kappa_1+\kappa$.

	Consider the generic function $\cup G \colon \kappa\to 2$. On the one hand,
	$$ \left(\left(k\circ j_{\vc}\right)(\cup G)\right)(k(\kappa)) = k( j_{\vc}(G)(\kappa) ) = k\left( 0 \right) =0 . $$
	On the other hand, 	since $\text{crit}(j^{M_{\wc}}_{U^*}) > \kappa_1\geq k(\kappa)$, 
	$$ \left(\left(j^{M_{\wc}}_{U^*}\circ j_\wc\right)(\cup G)\right)(k(\kappa))  = j_{\wc}(\cup G)(k(\kappa)) = 1  $$
	where we used the fact that $j_\wc(G) = G\times\{1\}\times G^*\times\{1\}\times G^{**}$ and $k(\kappa)\in \{\kappa, \kappa_1\}$. This shows that 	$\left(k\circ j_{\vc}\right)(G )\neq \left(j^{M_\wc}_{U^*}\circ j_{\wc}\right)( G )$, which is the desired contradiction.
\end{proof}

	We would like to present some intuition behind the proof that $\vc <_L \wc$ in Theorem \ref{Theorem: Separating Lipschitz from Ketonen the Easy way}. Given a normal measure $U$ on a measurable cardinal $\kappa$, $U$ is Lipschitz below $W$, where $W \equiv_{RK} U^2$ is the measure derived from $j_{U^2}$ using $j_U(\kappa) + \kappa$ as a seed. This can be proved in several ways; perhaps the simplest is the observation that $U <_k W$, and therefore $U <_L W$. Our underlying intuition, however, comes from the following winning strategy for Player I in the game $G_{\kappa}(W,U)$.
	
	Assume that $i < \kappa$, and that Players I and II have so far constructed sequences $\langle a(j) : j < i \rangle$ and $\langle b(j) : j < i \rangle$, respectively. Player I asks whether $i = i_1 + i_0$ for some inaccessible cardinals $i_0 < i_1$. Player I plays $a(i) = 0$ (intuitively indicating that the sequence $\langle a(j) : j < i \rangle$ does not “represent’’ a set in $W \equiv_{RK} U^2$) if and only if $b(i_0) = 1$ (intuitively indicating that the sequence $\langle b(j) : j < i_1 \rangle$ does “represent’’ a set in $U$). It is routine to verify that this gives Player I a winning strategy, establishing $U <_L W$.
	
	However, $U$ and $W$ are also Ketonen comparable. To separate the two orders, we adapt this intuition to the Ketonen-incomparable measures $\vc$ and $\wc$ in $V[G]$. The key point is that in the generic extension, the players gain access to new strategies once they recognize that their universe is a forcing extension via $\po$.\footnote{Specifically, such strategies may involve the computation of nice names and their values with respect to various generic sets.}
	
	More specifically, assume that $\langle a(j) : j < i \rangle$ and $\langle b(j) : j < i \rangle$ are the moves of Players I and II, respectively, in the game $G_{\kappa}(\wc,\vc)$ up to some $i < \kappa$, where $i$ has the form $i = i_1 + i_0$ for $i_0 < i_1$ inaccessibles. Consider the set
	$$Y = \{j<i_1 \colon  b(j)=1 \}\subseteq i_1.$$
	Player I can present $Y$ as the interpretation of a canonical\footnote{Here, “canonical’’ means that the antichains assembling the nice name are chosen least  with respect to a fixed-in-advance well order $\vartriangleleft$ of $V_\kappa$, as in the proof of Theorem  \ref{Theorem: Separating Lipschitz from Ketonen the Easy way}.} nice name. Player I then computes how this nice name would be evaluated with respect to the generic $G'$ obtained from $G$ by flipping the generic bit at coordinate $i_0$ (from $0$ to $1$ or vice versa). Player I plays $a(i) = 0$ if and only if $i_0$ belongs to the set obtained in this computation.
	
	This describes a winning strategy for Player I in the game $G_{\kappa}(\wc,\vc)$, and indeed the set $Z$ from the proof of Theorem \ref{Theorem: Separating Lipschitz from Ketonen the Easy way} is based on this strategy.
	
	As a final remark, we note that the set $Z$ from the proof of Theorem \ref{Theorem: Separating Lipschitz from Ketonen the Easy way} must not be a $\sigma$-complete ultrafilter (by Proposition \ref{Proposition: equivalent characterization of Lipschitz and Ketonen}). We conjecture, even though it's not really clear to us, that $Z$ is not even a filter.
	
	Finally, we conclude this section by proving the claim used in the proof of Theorem \ref{Theorem: Separating Lipschitz from Ketonen the Easy way}.

\begin{proof}[Proof of Claim \ref{Claim: HayutPovedaIntersectionOfImagesOfClubs}]
	We argue that
	\begin{center}
		$\bigcap\{ j_{U^m}(C)\subseteq\kappa \colon C \mbox{ is a club in }\kappa \}= \{ j_{U^n}(\kappa) \colon n< m \}.$
	\end{center}
	
	The inclusion $\supseteq$ is simple, so we concentrate on the other inclusion. Denote $\kappa_n = j_{U^n}(\kappa)$ for every $n<m$. Let $\alpha$ be an ordinal such that for every club $C\subseteq \kappa$, $\alpha\in j_{U^m}(C)$. We argue that this implies that $\alpha\in \{ \kappa, \kappa_1, \ldots, \kappa_{m-1} \}$. Assume otherwise, and let $1\leq n\leq m-1$ be such that $\alpha\in (\kappa_{n-1}, \kappa_n)$, where $\kappa_0$ denotes $\kappa$. We can write $\alpha= j_{U^n}(h)\left( \kappa, \kappa_1, \ldots, \kappa_{n-1} \right)$ for some $h\colon [\kappa]^{n-1}\to \kappa$, and we can further assume that for every $\vec{\xi} = \la \xi_0,\ldots, \xi_{n-1} \ra\in [\kappa]^{n-1}$, $h(\vec{\xi}) > \xi_{n-1}$. For every club $C\subseteq \kappa$, $j_{U^m}(C)\cap \kappa_{n} = j_{U^n}(C)$. Therefore, for every such $C$, the fact that $\alpha\in j_{U^m}(C)$ implies that--
	$$\{ \vec{\xi}\in [\kappa]^{n-1} \colon h(\vec{\xi})\in C  \}\in U^n.$$
	In particular, $\mbox{Im}(h)$ is stationary in $V$. We derive a contradiction by constructing a regressive function $\varphi\colon \mbox{Im}(h)\to \kappa$ which is not constant on any stationary subset of $\kappa$. Define for every $x\in \text{Im}(h)$,
	$$ \varphi(x) = \min\{  \eta< \kappa \colon \mbox{for some }  \vec{\xi} = \la \xi_0,\ldots, \xi_{n-1} \ra\in [\kappa]^{n-1},\   h(\vec{\xi}) = x \mbox{ and } \xi_{n-1} = \eta \}. $$
	Note that for every $x\in \mbox{Im}(h)$, $\varphi(x)< x$. Also, for every $\eta< \kappa$, $\varphi^{-1}[\{ \eta \}] \subseteq h[\  [\eta+1]^{n-1}]$. Thus, $\varphi^{-1}[\{ \eta \}]$ is bounded in $\kappa$ and in particular is nonstationary.
\end{proof}

\section{Nonstationary support products}

Our next goal is to strengthen Theorem \ref{Theorem: Separating Lipschitz from Ketonen the Easy way} by additionally requiring that the Weak Ultrapower Axiom holds. By Corollary \ref{Corollary: UA implies that Ketonen is Lipschitz}, the full Ultrapower Axiom must fail in the resulting model. To achieve this, we rely on the recent construction of a model of Weak UA + $\neg$UA from \cite{kaplan2025number}. The main technique relevant for our purposes is nonstationary-support product forcing. A further technical feature we introduce is a substitute for “nice names’’ adapted to the setting of nonstationary-support products.

\begin{definition}
	\begin{enumerate}
		\item A set $A$ is called nonstationary in inaccessibles if for every inaccessible cardinal $\lambda$, $A\cap \lambda$ is nonstationary in $\lambda$. 
		\item Assume that $\la \qo_\alpha\colon \alpha<\kappa \ra$ are posets. The nonstationary support product $\prod^{NS}_{\alpha<\kappa} \qo_\alpha$ consists of conditions which are functions $p$ with domain $\alpha$, such that:
		\begin{enumerate}
			\item for every $\beta<\alpha$, $p(\beta)\in {\qo}_\beta$.
			\item the set $ \{  \beta<\alpha \colon p(\beta) \neq {0}_{{\qo}_\beta} \}$ is nonstationary in inaccessibles.
		\end{enumerate}
		
	\end{enumerate}
	
\end{definition}

The most central tool for analyzing nonstationary support products is the following fusion lemma, whose proof can be found in \cite[Lemma 1.3]{kaplan2025number}

\begin{lemma}(Fusion Lemma)\label{Lemma: Fusion for a product}
		Let $\kappa$ be a limit of inaccessible cardinals, and let $I\subseteq \kappa$ be an unbounded set of inaccessibles below $\kappa$. Let $\po = \prod_{\alpha\in I}^{NS}\qo_\alpha$ be a nonstationary support product. Assume that:
		\begin{enumerate}
			\item for every $\alpha\in I$, $\text{rank}({\qo}_\alpha)< \min(I\setminus \alpha+1)$.
			\item for every $\alpha\in I$, $\qo_\alpha$ is $\alpha$-closed. 
		\end{enumerate}
		Then $\po$ satisfies the \textbf{Fusion Property}; that is, given:
		\begin{itemize}
			\item a condition $p\in \po$,
			\item a sequence $\la d(\alpha) \colon \alpha<\kappa \ra$ of dense-open subsets of $\po$,
		\end{itemize}
		there exists $p^*\geq p$ and a club $C\subseteq \kappa$ (if $\cf(\kappa)= \omega$, $C$ is an unbounded cofinal $\omega$-sequence) such that, for every $\alpha\in C$, the set
		$$ \{ r\in \po\uhr {\alpha+1} \colon r^{\frown} p^*\setminus (\alpha+1)\in d(\alpha)  \} $$
		is dense in $\po\uhr {\alpha+1}$ above $p^*\uhr \alpha+1$.\footnote{We remark that $\po\uhr (\alpha+1)$ can be naturally identified with the poset $(\prod^{NS}_{\beta<\alpha} \qo_\beta)\times \qo_\alpha$. Also, each condition $p\in G$ can be identified with the pair $(p\uhr \alpha+1, p\setminus \alpha+1)\in (\po\uhr (\alpha+1))\times (\po\setminus (\alpha+1))$, where $p\setminus (\alpha+1)$ is defined to be $p\uhr (\kappa\setminus (\alpha+1))$.}
\end{lemma}

	Recall that our main goal is to mimic the proof of Theorem \ref{Theorem: Separating Lipschitz from Ketonen the Easy way} and produce two measures that are Lipschitz comparable but Ketonen incomparable. The main obstacle is the lack of a convenient analogue of “nice names” for subsets of $\kappa$ in the setting of nonstationary-support products. Nice names themselves are not suitable here, since they typically do not belong to $H_{\kappa^+}$ (as a consequence of the failure of the $\kappa^+$-c.c.\ in nonstationary-support products). Fortunately, a trick that effectively replaces nice names already appears in several works on nonstationary-support iterations. 
	
	We assume throughout that $\po$ is a forcing notion satisfying the assumptions of the Fusion Lemma \ref{Lemma: Fusion for a product}. The goal is to code each element $X\in (\mathcal{P}(\kappa))^{V[G]}$ by a sequence $\vec{\tau}\in ({V_\kappa}^\kappa)^V$, and develop an 'interpretation procedure' that retrieves $X$ in $V[G]$ from $\vec{\tau}$.

	The main idea is that $\dot{X}\cap \alpha\in V^{\po_{\alpha+1}}$, so we may find a $\po_{\alpha+1}$-name $\tau_\alpha\in V_\kappa$ for $\dot{X}\cap \alpha$. This was done above for each $\alpha$ separately, but the Fusion Lemma allows to find a single condition $p$, a club $C\subseteq \kappa$ and a sequence $\la \tau_\alpha \colon \alpha<\kappa \ra\in V$ such that $p\Vdash \dot{X}\cap \alpha = (\tau_\alpha)_{\dot{G}_{\alpha+1}}$ for every $\alpha\in C$ (see Lemma \ref{Lemma: every subset of kappa has a fusion code} below). By picking such $p$ inside $G$, the associated sequence $\vec{\tau} = \la \tau_\alpha \colon \alpha<\kappa  \ra$ can serve as a substitute for the name $\dot{X}$, by noting that
	$$ (\dot{X})_G = \bigcup_{\alpha\in C} \left( \tau_\alpha \right)_{G_{\alpha+1}}.$$
	The advantage is that $\vec{\tau}\in H_{\kappa^+}$, so $\vec{\tau}$ makes it into any inner model of $V$ that contains $H_{\kappa^+}$. Indeed, A typical application of this trick is the proof that whenever $j\colon V \to M$ is, say, an ultrapower embedding by a normal measure on $\kappa$, and $G \subseteq \po = \prod^{NS}_{\alpha \in I} \qo_\alpha$ is generic over $V$, then $(\mathcal{P}(\kappa))^{V[G]} = (\mathcal{P}(\kappa))^{M[G]}$, since every subset of $\kappa$ can be coded by a sequence $\la \tau_\alpha \colon \alpha<\kappa \ra\in H_{\kappa^+}$. 
	
	The argument above already appears in  \cite{FriMag09,ApterCummingsNormalMeasuresOnTallCards,benneriaunger2017homogeneouschangesincofinalities} and many other works involving the nonstationary support iterations or products.

	\begin{definition}
		Assume the settings of Lemma \ref{Lemma: Fusion for a product}. Fix a sequence $\vec{\tau} = \la \tau_\alpha \colon \alpha<\kappa\ra\in V$. Let $G\subseteq \po$ be generic over $V$ and  $X\subseteq \left(V_\kappa\right)^{V[G]}$. We say that $\vec{\tau}$ codes $X$ via $G$, and denote--
		$$ X = \mbox{val}(\vec{\tau}, G) $$ 
		if there exists a club $C\subseteq \kappa$ such that for every $\alpha\in C$, $\tau_\alpha$ is a $\po_{\alpha+1}$-name such that $$X\cap \left(V_\alpha\right)^{V[G]} = \left(\tau_\alpha\right)_{G_{\alpha+1}}.$$
	\end{definition}

	\begin{remark}
		Note that $\mbox{val}(\vec{\nu},  G)$ does not depend on the club $C$, in the sense that, if there are clubs $C,D\subseteq \kappa$ and $X,Y\subseteq \left( V_{\kappa} \right)^{V[G]}$ such that for every $\alpha\in C$,
		$$ X\cap \left( V_{\alpha} \right)^{V[G]} = \left(\tau_\alpha\right)_{G_{\alpha+1}} $$
		and for every $\xi \in D$,
		$$ Y\cap \left( V_{\alpha} \right)^{V[G]} = \left(\tau_\alpha\right)_{G_{ \alpha+1}} $$
		then $X = Y$.
	\end{remark}

	\begin{lemma}\label{Lemma: every subset of kappa has a fusion code}
		Assume the settings of Lemma \ref{Lemma: Fusion for a product}. Suppose that  $\name{X}$ is a $\po$-name for a subset of $\left(V_{\kappa}\right)^{V^{\po}}$. Then there exists $p\in G$, a club $C\subseteq \kappa$ (in $V$), and a sequence $\vec{\tau} = \la {\tau}_\alpha \colon \alpha\in C \ra\in V$, such that, for every $\alpha\in C$, ${\tau}_\alpha$ is a $\po_{\alpha+1}$-name, and
		$$ p\Vdash \name{X}\cap \left(V_{\alpha}\right)^{  V[\name{G}_{\alpha+1}] } = \left( \tau_{\alpha} \right)_{G_{\alpha+1}}. $$
		In particular, in $V[G]$, $\left(\dot{X}\right)_G = \text{val}( \vec{\tau}, G ).$
	\end{lemma}
	
	\begin{proof}
		Define the sequence of dense open sets $\la d(\alpha) \colon \alpha<\kappa  \ra$, where, for each $\alpha<\kappa$,
		$$ d(\alpha) = \{ q\in \po \colon q\uhr \alpha+1 \Vdash \exists Y\subseteq  \left( V_\alpha \right)^{V[\name{G}_{\alpha+1}]},  q\setminus \alpha+1\Vdash \name{X}\cap \left(V_\alpha\right)^{V[\name{G}_{\alpha+1}]} = Y \}. $$
		Note that the density of $d(\alpha)$ follows since $\po\setminus \alpha+1$ is forced to be more than $\alpha$-closed. By applying the fusion lemma  \ref{Lemma: Fusion for a product}, we may find a condition $p\in G$ and a club $C\subseteq \kappa$, such that, for every $\alpha\in C$, the set 
		$$D(\alpha) =  \{  r\in \po_{\alpha+1}  \colon r\Vdash \exists Y\subseteq   \left( V_\alpha \right)^{V[\name{G}_{\alpha+1}]},  \  {p\setminus \alpha+1}\Vdash \name{X}\cap \left(V_\alpha\right)^{V[\name{G}_{\alpha+1}]} = Y  \} $$
		is dense open in $\po_{\alpha+1}$. 
		
		Fix a maximal antichain $B(\alpha)\subseteq D(\alpha)$ above $p\uhr \alpha+1$, and for each $r\in B(\alpha)$, let ${\tau}^r_\alpha$ be a $\po_{\alpha+1}$-name for the subset $Y$ as above. The $\po_{\alpha+1}$ mix-name for $\la {\tau}^r_\alpha \colon r\in B(\alpha)\ra$ is forced by $p_{\alpha+1}$ to be a subset of $\left( V_\alpha \right)^{V^{\po_{\alpha+1}}}$, where, by “mix-name”, we mean the name 
		$$\tau_\alpha =  \{ \la r, {\tau}^r_\alpha \ra  \colon r\in B(\xi) \}\cup \{ \la r, \check{0} \ra \colon r \text{ doesn't extend } p_{\alpha+1}  \}.$$ 
		Then for every $\alpha\in C$, $\tau_\alpha$ is a $\po_{\alpha+1}$-name, and 
		$$ p\Vdash \name{X}\cap \left( V_\alpha \right)^{V[\name{G}_{\alpha+1}]} = \left(  \tau_\alpha \right)_{ \name{G}_{\alpha+1} } $$
		as desired.
	\end{proof}
	
	\begin{corollary}
		Let $\kappa$ and $\po = \prod^{NS}_{\alpha\in I} \qo_\alpha$ be as in Lemma \ref{Lemma: Fusion for a product}. Assume that $M$ is an inner model of $V$ with $H_{\kappa^+}\subseteq M$. Let $\po^M = \left(\prod^{NS}_{\alpha\in I} \qo_\alpha\right)^M$. Then:
		\begin{enumerate}
			\item $\po = \po^M$.
			\item If $G\subseteq \po$ is generic over $V$, then $G\subseteq \po = \po^M$ is generic over $M$.
			\item $(\mathcal{P}(\kappa))^{V[G]} = (\mathcal{P}(\kappa))^{M[G]} $.
		\end{enumerate} 
	\end{corollary}
	
	\begin{proof}
		The facts that $V,M$ agree on ${}^\kappa\left( V_\kappa \right)$ and $\left( \mbox{Cub}_\kappa \right)^{V} = \left( \mbox{Cub}_\kappa \right)^M$, imply  that the nonstationary support product $\prod^{NS}_{\alpha<\kappa} \qo_\alpha$ is correctly computed in $M$. Next, since $M$ is an inner model of $V$, any $G\subseteq \po$ which is generic over $V$, is also generic over $M$. Thus, we concentrate on proving that $(\mathcal{P}(\kappa))^{V[G]} = (\mathcal{P}(\kappa))^{M[G]} $. Assume that $X\in V[G]$ is a subset of $\kappa$, and let $\dot{X}$ be a $\po$-name for it. Apply Lemma \ref{Lemma: every subset of kappa has a fusion code} to find a condition $p\in G$, a club $C\subseteq \kappa$ and a sequence $\la \tau_\alpha \colon \alpha\in C \ra\in {}^\kappa\left( V_\kappa \right)$ such that 
		$$ p\Vdash \dot{X}\cap \alpha = \left( \tau_\alpha \right)_{\dot{G}_{\alpha+1}}. $$
		Since $p, C, \vec{\tau}\in M$, 
		$$ X = \bigcup_{ \alpha\in C } \left( \tau_\alpha \right)_{G_{\alpha+1}} $$
		and the above computation can be done in $M[G]$.
	\end{proof}

	\section{Separating the orders under Weak UA}
	We proceed and define our main forcing $\po$. 
	Let $\kappa$ be a measurable cardinal, and denote 
	$$I = \{ \alpha<\kappa \colon \alpha \text{ is inaccessible} \}.$$
	Let $\la \po_\alpha, \dot{\qo}_\alpha \colon \alpha<\kappa \ra$ be a nonstationary support product $\prod^{NS}_{\alpha\in I} \qo_\alpha$ in which, for each inaccessible cardinal $\alpha<\kappa$, $\qo_\alpha = \{  0_{\qo_\alpha}, 0,1 \}$, where $0,1$ are incompatible elements (as above, we simply denote each forcing $\qo_\alpha$ by $\qo$). For every other value of $\alpha<\kappa$, $\qo_\alpha$ is the trivial forcing. Let $\po = \po_\kappa$. In other words,
	$$ \po = \{ f\colon X\to 2 \colon X\subseteq I \text{ is nonstationay in inaccessibles} \}, $$
	ordered by inclusion. 
	
	The following is one of the central results of \cite{kaplan2025number}.
	
	\begin{theorem}(\cite[Theorem 0.4]{kaplan2025number})\label{theorem: a nontrivial model of weak UA}
		Assume $V= L[U]$. Then $V[G]\vDash $Weak UA$+\neg$UA.
	\end{theorem}

	\begin{claim}\label{Claim: two generics after nonstat supp product}
		Let $\kappa$ be a measurable cardinal and $U$ a normal measure on $\kappa$. Let $\po$ be the above forcing and $G\subseteq \po$ generic over $V$. 
		\begin{enumerate}
			\item for every $i\in \{0,1\}$, let
			$$H_i = \{ q\in j_U(\po)\colon \exists p\in G \left( q\leq j_U(p)\cup \{ (\kappa,i)  \} \right) \}.$$
			Then each $H_i$ is $j_U(\po)$-generic over $M_U$, and $j_U[G]\subseteq H_i$.
			\item for every $i\in \{0,1\}$, there exists a normal measure $U_i\in V[G]$ on $\kappa$ such that $j^{V[G]}_{U_i}\colon V[G]\to M_U[H_i]$ is an elementary embedding that extends $j_U$ and maps $G$ to $H_i$.
		\end{enumerate}
	\end{claim}
	
	\begin{proof}(Sketch; we refer the interested reader to \cite[Theorem 0.2]{kaplan2025number} for a more detailed proof)
		It's not hard to verify that each $H_i$ is a filter. For genericity, fix $D\in M_U$ a dense open subset of $j_U(\po)$. Let $\alpha\mapsto d(\alpha)$ be a function in $V$ such that each $d(\alpha)\subseteq \po$ is dense open, and $D = [\alpha\mapsto d(\alpha)]_U$. By the Fusion Lemma \ref{Lemma: Fusion for a product}, there exists a condition $p\in G$ and a club $C\subseteq \kappa$ such that for every $\alpha\in C$, 
		$$  \{  q\in \po_{\alpha+1} \colon q\cup (p\setminus \alpha+1)\in d(\alpha) \} $$
		is a dense subset of $\po\uhr{\alpha+1}$. In particular, since $C\in U$, 
		$$ \{ q\in j_U(\po)\uhr (\kappa+1)\colon   q\cup (j_U(p)\setminus \kappa+1)\in D \} $$
		is a dense subset of $j_U(\po)\uhr (\kappa+1) = \po\times \qo$. Since $G\times \{ i \}$ is generic for $\po\times \qo$ over $M_U$, by extending $p$ inside $G$, we deduce that $j_U(p)\cup \{ (\kappa,i) \}\in D\cap H_i$, as desired. 
		
		Finally, it's not hard to verify that $j_U[G]\subseteq H_i$ for each $i\in \{0,1\}$, and $j_U\colon V\to M_U$ lifts in two distinct ways to an embedding whose domain is $V[G]$, by mapping $G$ to $H_i$. In fact, each lifted embedding is the ultrapower embedding associated with the normal measures derived from it using $\kappa$ as a seed. Let $U_0, U_1$ be those normal measures.
	\end{proof}
	
	We remark that, assuming GCH, forcing with $\po$ preserves cardinals (see \cite[Corollary 1.6]{kaplan2025number}). 
	
	We are now prepared for the proof of Theorem \ref{theorem: intro}.
	
	\begin{proof}[Proof of Theorem \ref{theorem: intro}]
		Assume that $V = L[U]$, and denote by $\vartriangleleft$ the canonical well order of $H_{\kappa^+}$ in $L[U]$. Also denote $\kappa_1 = j_U(\kappa)$. Let $\po$ be the forcing nation defined at the beginning of this section, and let $G\subseteq \po$ be generic over $L[U]$. By Theorem \ref{theorem: a nontrivial model of weak UA}, $V[G]$ is a model of Weak UA and $\neg$UA. Thus, it remains to show that there are measures $\vc,\wc\in V[G]$ on $\kappa$ which are Lipschitz comparable and Ketonen incomparable.
		
		Let $\vc = U_0$ and $\wc = \{ X\subseteq \kappa \colon \kappa_1+\kappa\in j_{(U_1)^2}(X) \}$ (where $U_0, U_1\in V[G]$ are the normal measures on $\kappa$ from Claim \ref{Claim: two generics after nonstat supp product}). As in the proof of Theorem \ref{Theorem: Separating Lipschitz from Ketonen the Easy way}, we argue that $\vc <_L \wc$ but $\vc,\wc$ are Ketonen incomparable. The proof that $\vc, \wc$ are Ketonen incomparable is identical to the argument from Theorem \ref{Theorem: Separating Lipschitz from Ketonen the Easy way}, so we concentrate on proving that $\vc <_L \wc$.
		
		First, recall that $\wc$ is Rudin-Keisler equivalent to $(U_1)^2$, since the seed $(\kappa_1,\kappa)$ of $(U_1)^2$ can be extracted from the ordinal $\kappa_1+\kappa$ and vice-versa. In particular, $\text{Ult}(V[G], W) = \text{Ult}(V[G], (U_1)^2)$, and it has the form $M^* = M_{U^2}[G\times \{ 1\}\times G^*\times \{1\}\times G^{**}]$, where $G^*$ is generated by
		$$ \{  j_U(p)\setminus (\kappa+1) \colon p\in G \} $$
		and $G^{**}$ is generated by
		$$ \{ j_{U^2}(p)\setminus (\kappa_1+1) \colon p\in G \} .$$		
		
		Fix a subset $Y$ of $\kappa_1$ in $M^*$. Since $j_{U^2}(\po)\setminus (\kappa_1+1)$ is more than $\kappa_1$-closed, $Y\in M_{U^2}[G\times\{1\}\times G^*]$. 
		In $M_{U^2}[G\times\{1\}\times G^*]$, let $\vec{\tau}_Y$ be the $j_{U^2}(\vartriangleleft)$-least sequence in $(H_{(\kappa_1)^+})^{M_{U^2}}$ such that 
		$$M_{U^2}[G\times\{1\}\times G^*] \vDash Y = \text{val}\left( \vec{\tau}_Y, G\times\{1\}\times G^* \right).$$
		Note that such $\vec{\tau}_Y$ exists by Lemma \ref{Lemma: every subset of kappa has a fusion code}. 
		
		Define, in  $M_{U^2}[ G\times \{ 1 \}\times G^* ]$,  the set
		$$Z = \{ Y\subseteq \kappa_1 \colon \text{val}(\vec{\tau}_Y, G\times\{0\}\times G^*)  \text{ is defined and }\kappa\in  \text{val}(\vec{\tau}_Y, G\times\{0\}\times G^*) \}.$$
		The definition of $Z$ can be carried out in $M_{U^2}[G\times\{1\}\times G^*]$, since the assignment of $\vec{\tau}_Y$ for a set $Y$ can be done internally in $M_{U^2}[G\times \{1\}\times G^*]$. It's clear from the definition that $Z$ concentrates on $\kappa_1< [Id]_{\wc}$.\footnote{Note that $\kappa_1 = \text{val}\left( \la \check{\alpha} \colon \alpha<\kappa \ra, H \right)$ for any generic $H\subseteq j_U(\po)$ over $M_{U}$ (or over $M_{U^2}$).} 
		
		We argue that for every $X\subseteq \kappa$, 
		$$X\in \vc\iff j_\wc(X)\cap \kappa_1\in Z.$$
		Indeed, fix $X\subseteq \kappa$. Let $\vec{\tau}\in H_{\kappa^+}$ be the $\vartriangleleft$-least such that, in $V[G]$, $X = \text{val}(\vec{\tau}, G)$. By elementarity, it follows that both $\text{val}( j_U(\vec{\tau}), G\times \{ 0 \}\times G^* )$ and  $\text{val}( j_U(\vec{\tau}), G\times \{ 1 \}\times G^* )$ are defined; the former is equal to $j_{U_0}(X)$ and the latter to $j_{U_1}(X)$. Also, for every $i\in \{0,1\}$,
		$$\text{Ult}(V[G], U_i) \vDash j_U(\vec{\tau}) \text{ is } j_U(\vartriangleleft)\text{-minimal such that }j_{U_i}(X) = \text{val}( j_U(\vec{\tau}), G\times \{ i \}\times G^* ). $$
			
		
		Denote $Y = j_{\wc}(X)\cap \kappa_1$. We argue that, in the above notations, $ \vec{\tau}_{ Y } = j_U(\vec{\tau})$. In other words, 
		$$ M_{U^2}[G\times \{1\}\times G^*]\vDash j_{U}(\vec{\tau}) \text{ is }j_{U^2}(\vartriangleleft)\text{-least such that } Y = \text{val}\left( j_U(\vec{\tau}) ,G\times \{1\}\times G^* \right). $$
		
		First, the fact that 
		$$M_{U^2}[G\times \{1\}\times G^*]\vDash Y = \text{val}\left( j_U(\vec{\tau}) ,G\times \{1\}\times G^* \right)$$ 
		follows from the fact that, in $M_{U}[G\times \{1\}\times G^*]$,
		
		$$Y = j_{U_1}(X) = j_{U_1}\left( \text{val}( \vec{\tau}, G )  \right) = \text{val}(j_U(\vec{\tau}), G\times \{1\}\times G^* ). $$
		Note that it's important here that $\text{val}(j_U(\vec{\tau}), G\times \{1\}\times G^* )$ is computed the same way in $M_U[G\times \{1\}\times G^*]$ and $M_{U^2}[G\times \{1\}\times G^*]$. Finally, the $j_{U^2}(\vartriangleleft)$-minimality of $j_U(\vec{\tau})$ among the set of codes for $Y$ in $M_{U^2}[G\times\{1\}\times G^*]$ follows from the fact that $j_{U^2}(\vartriangleleft)$ coincides with $j_{U}(\vartriangleleft)$ on $\left( H_{(\kappa_1)^+} \right)^{M_U}$, $j_U(\vec{\tau}) \in \left( H_{(\kappa_1)^+} \right)^{M_U}$, and $j_U(\vec{\tau})$ is, by elementarity, the $j_U(\vartriangleleft)$-minimal code for $Y\cap \kappa_1 = j_{U_1}(X)$. 
		
		Overall, we deduce that indeed $\vec{\tau}_Y = j_U(\vec{\tau})$, so
		\begin{align*}
			X\in \vc \iff & \kappa\in j_{\vc}(X) = j_{U_0}(X) \\
					 \iff & \kappa \in \text{val}\left( j_U(\vec{\tau}),  G\times \{ 0 \}\times G^* \right) \\
					 \iff & \kappa\in \text{val}\left( \vec{\tau}_Y,  G\times \{ 0 \}\times G^* \right)  \\
					 \iff &j_\wc(X)\cap \kappa_1\in Z.
		\end{align*}
		By Proposition \ref{Proposition: equivalent characterization of Lipschitz and Ketonen}, this implies $\vc<_L \wc$.
	\end{proof}
	
	We conclude this paper with the following open problem, raised by Goldberg:
	
	\begin{question}[Goldberg]\label{Question: Linearity of Lipschitz implies UA?}
		Does UA follow from the assumption that the Lipschitz order is linear on the class of $\sigma$-complete ultrafilters?
	\end{question}
	
	The Mitchell order is not linear in the models constructed in the proofs Theorems \ref{Theorem: Separating Lipschitz from Ketonen the Easy way}, \ref{theorem: intro}. Thus, those models don't satisfy the linearity of the Lipschitz order, and cannot settle Question \ref{Question: Linearity of Lipschitz implies UA?}.  
	
	Question \ref{Question: Linearity of Lipschitz implies UA?} is motivated by the fact that linearity of the Ketonen order is equivalent to UA (see Goldberg's  Theorem~\ref{Theorem: Goldberg, UA is equivalent to the linearity of Ketonen} above). A positive answer would suggest that UA is a determinacy principle.

\bibliographystyle{plain}
\bibliography{Bibliography}

\end{document}